\documentclass[12pt,a4paper]{amsart}
\usepackage{amsfonts,color}
\usepackage{amsthm}
\usepackage{amsmath}
\usepackage{amscd}
\usepackage[utf8]{inputenc}
\usepackage{t1enc}
\usepackage[mathscr]{eucal}
\usepackage{indentfirst}
\usepackage{graphicx}
\usepackage{graphics}
\usepackage{pict2e}
\usepackage{epic}
\usepackage{url}
\usepackage{epstopdf}
\usepackage{thmtools}
\usepackage{thm-restate}
\usepackage{comment}

\numberwithin{equation}{section}
\usepackage[margin=2.6cm]{geometry}

\theoremstyle{plain}

\newtheorem{Th}{Theorem}[section]
\newtheorem{Lemma}[Th]{Lemma}
\newtheorem{Cor}[Th]{Corollary}

 \theoremstyle{definition}
\newtheorem{Def}[Th]{Definition}

\newtheorem{Rem}[Th]{Remark}
\newtheorem{?}[Th]{Problem}

\newcommand{\C}{\mathbb{C}}
\newcommand{\E}{\mathbb{E}}
\renewcommand{\P}{\mathbb{P}}

\usepackage{mathtools}
\usepackage{hyperref}
\hypersetup{
    colorlinks,
    citecolor=black,
    filecolor=black,
    linkcolor=black,
    urlcolor=black
}
\usepackage{cleveref}
\DeclarePairedDelimiterX{\inner}[2]{\langle}{\rangle}{#1, #2}
\makeatletter
\newcommand{\mainsectionstyle}{%
  \renewcommand{\@secnumfont}{\bfseries}
  \renewcommand\section{\@startsection{section}{1}%
	\z@{.7\linespacing\@plus\linespacing}{.5\linespacing}%
	{\normalfont\scshape\centering\bfseries}}

}
\makeatother
\makeatletter
\newcommand{\bigcomp}{%
  \DOTSB
  \mathop{\vphantom{\sum}\mathpalette\bigcomp@\relax}%
  \slimits@
}
\newcommand{\bigcomp@}[2]{%
  \begingroup\m@th
  \sbox\z@{$#1\sum$}%
  \setlength{\unitlength}{0.9\dimexpr\ht\z@+\dp\z@}%
  \vcenter{\hbox{%
    \begin{picture}(1,1)
    \bigcomp@linethickness{#1}
    \put(0.5,0.5){\circle{1}}
    \end{picture}%
  }}%
  \endgroup
}
\newcommand{\bigcomp@linethickness}[1]{%
  \linethickness{%
      \ifx#1\displaystyle 2\fontdimen8\textfont\else
      \ifx#1\textstyle 1.65\fontdimen8\textfont\else
      \ifx#1\scriptstyle 1.65\fontdimen8\scriptfont\else
      1.65\fontdimen8\scriptscriptfont\fi\fi\fi 3
  }%
}

\makeatother

 \newcommand{\Cap }{\mathrm{cap}}

\begin{document}

\title{Capacity of Lorentzian polynomials and distance to binomial distributions}

\author{\'Ad\'am Schweitzer}
\thanks{The author thanks Péter Csikvári for his help as supervisor.}

\begin{abstract}
In this paper we study the capacity of Lorentzian polynomials. We give a new proof of a theorem of Br\"and\'en, Leake and Pak. Our approach is probabilistic in nature and uses a lemma about a certain distance of binomial distributions to distributions with fixed expected value. 
    
\end{abstract}


\maketitle

\section{Introduction}

Lorentzian polynomials form an important class of multivariate polynomials that in some sense generalize the class of real stable homogeneous polynomials (\cite{branden2020lorentzian}). Stable polynomials are defined as follows: a multivariate polynomial $P\in \C[x_1,\dots x_n]$ is stable if $P(z_1,\dots ,z_n)\neq 0$ whenever $\mathrm{Im}(z_i)>0$ for all $i\in \{1,2,\dots ,n\}$. We call $P$ a real stable polynomial if $P$ is stable and all coefficients are real. Univariate real stable polynomials can be easily seen to be exactly the real-rooted polynomials. Thus real stable polynomials can be viewed as a multivariate generalization of real-rooted polynomials. A homogenous bivariate polynomial is real stable if and only if, if we substitute 1 to one of it's variables we get a real rooted polynomial (Example 2.3 in \cite{branden2020lorentzian}). The coefficients of a real rooted polynomials satisfy the Newton inequality:
$$\frac{a_i^2}{\binom{n}{i}^2}\geq \frac{a_{i-1}}{\binom{n}{i-1}}\frac{a_{i+1}}{\binom{n}{i+1}}.$$
A homogenous bivariate polynomial is Lorentzian if and only if its coefficients satisfy this inequality (Example 2.3 in \cite{branden2020lorentzian}). It turns out that the class of stable polynomials is closed under coordinate-wise differentiation. 

This leads to a definition of Lorentzian polynomials for all degrees:
A homogeneous polynomial of degree at least 3 is Lorentzian if its support is M-convex in the sense of \cite{murota1998discrete}, and each of its derivatives are Lorentzian as well.
A homogeneous polynomial of degree at most 2 is Lorentzian if it is real stable and has non negative coefficients.

That this definition is equivalent to the usual is proved in \cite{branden2020lorentzian} in Theorem 2.25.



Another important concept for multivariate polynomials with real coefficients is the capacity. It is defined as follows: let $P\in \mathbb{R}_{\geq 0}[x_1,\dots ,x_n]$ and $\underline{\alpha}\in \mathbb{R}_{\geq 0}^n$:
$$\mathrm{cap}_{\underline{\alpha}}(P)=\inf_{x_1,\dots ,x_n>0}\frac{P(x_1,\dots ,x_n)}{x_1^{\alpha_1}\dots x_n^{\alpha_n}}.$$
Capacity was introduced by Gurvits (\cite{gurvits2007van}) and its main power arises from the fact that when $P$ is a stable polynomial with non-negative coefficients and $T$ is some operator on polynomials preserving stability, then it is often possible (\cite{leake2018counting}) to prove an inequality of type
$$\mathrm{cap}_{\underline{\beta}}(TP)\geq c(T,\underline{\alpha},\underline{\beta})\mathrm{cap}_{\underline{\alpha}} (P).$$
Gurvits used this observation to prove a series of results including a new proof of van der Waerden conjecture on permanents(\cite{Gurvits2006HyperbolicPA}), a new proof of a theorem of Schrijver on the number of perfect matchings of regular bipartite graphs (\cite{SCHRIJVER1998122},\cite{gurvits2007van}), a proof of the asymptotic lower matching conjecture (\cite{gurvits2011unleashing}), Bapat's conjecture.
In the heart of this capacity inequality there is an inequality about (univariate) real-rooted polynomials. There are several proofs for this inequality. One of them reveals a connection with probability theory (\cite{csikvari2020short}), namely, a theorem of Hoeffding (\cite{hoeffding1956} Theorem 5,\cite{branden2020lower} Corollary 5.9) on the probability distribution generated by the coefficients of the polynomial easily implies the required statement. 
The corresponding statement for Lorentzian polynomials was proved by Br\"and\'en, Leake and Pak (\cite{branden2020lower}). There does not seem to any connection with probability theory. In this paper we give a proof that has a probabilistic nature. The idea is that if a random variable $X$ has binomial distribution $\mathrm{Bin}(n,p)$ and we condition on an event $A$ such that the conditioned random variable has an expected value $ns$ different from $np$, then one can give a bound on the probability $A$. (The precise statement is Lemma~\ref{main lemma}.) 

\section{Capacity of Lorentzian polynomials}

In this section we prove the following statement.

\begin{Th} \label{capacity}
If $P$ is a Lorentzian polynomial,  then
$$\Cap  _{\underline{\alpha}}(P) \cdot  \binom{n}{k} \bigg( \frac{k}{n} \bigg)^k \bigg( \frac{n-k}{n} \bigg)^{n-k} \leq \frac{1}{k!} \Cap_{\underline{\alpha}^*}\bigg(\frac{\partial P}{\partial x_i^k}\bigg|_{x_i=0}\bigg)  $$
where n is the total degree of polynomial and $\alpha_i=k$ and $\underline{\alpha}^*$ denotes $\underline{\alpha}$ vector without $\alpha_i$.
\end{Th}

This implies the following useful property of Lorentzian polynomials.

\begin{Cor}
Let $P(x_1,\dots ,x_n)$ be a Lorentzian polynomial. Suppose that the total degree of $P$ is $d$. Let $\underline{r}=(r_1,\dots ,r_n)$ and $a_{\underline{r}}$ be the coefficient of $\prod_{i=1}^nx_i^{r_i}$ in $P$. Then
$$a_{\underline{r}}\geq \prod_{i=1}^n\binom{d}{r_i}\bigg(\frac{r_i}{d}\bigg)^{r_i}\bigg(\frac{d-r_i}{d}\bigg)^{d-r_i}\mathrm{cap}_{\underline{r}}(P).$$
\end{Cor}

The key statement to prove Theorem~\ref{capacity} is to prove the following theorem on univariate polynomials. To state the theorem we will need the following definition.

\begin{Def}[Pólya frequency sequence of order two]
A sequence $a=(a_1,a_2\ldots a_n)$ is defined as a Pólya frequency sequence of order two if and only if it satisfies the following conditions:
\begin{itemize}
    \item Every element of $a$ is non-negative.
    \item If $a_i>0$ and $a_j>0$ then for all $k$ such that $i<k<j$ it holds that $a_k>0$.
    \item For all $i$ it holds that $a_i^2\geq a_{i-1}a_{i+1}$
\end{itemize}
\end{Def}

Now we can state the needed theorem.

\begin{Th}
Let $P(x)=\sum_{j=0}^n a_jx^j$, where $\frac{a_i}{\binom{n}{i}}$ forms a Pólya frequency sequence of order two. Let $P(1)=1$ and $P'(1)=ns \in \mathbb{N}$. In this case:
$$a_{ns}\geq \binom{n}{ns} (s^s \left(1-s\right)^{(1-s)})^n.$$\label{thm}
\end{Th}
\begin{proof}
Let us define the sequence $b_i$ as $b_i:=\frac{a_i}{\binom{n}{i}}$. This way $b_i$ must form a log-concave series.\\
Let $p\in (0,1)$ be determined by the equation $\frac{p}{1-p}=\frac{b_{ns}}{b_{ns-1}}$.
Let $q=\frac{p}{1-p}$.
Let us choose a number $c$ such that: $$b_{ns}=c (p^{s} (1-p)^{1-s})^n,$$  $$b_{ns-1}=c (p^{s-\frac{1}{n}} (1-p)^{1-s+\frac{1}{n}})^n.$$
As $b_i$ is positive, this can be done.
\smallbreak
From the log concavity of $b_i$ we also know that
\[
\frac{b_{i+1}}{b_{i}}\geq \frac{b_{i}}{b_{i-1}}\tag{1} \label{eq:logcon}.
\]
\smallbreak
From these facts we can prove the following bound for $a_i$
\begin{Lemma}
For every $i\in [0,n]$
$$a_i\leq \binom{n}{i} c p^i (1-p)^{n-i}.$$
\end{Lemma}
\begin{proof}

As $\frac{b_{ns}}{b_{ns-1}}=q$ it is easy to see from \eqref{eq:logcon}, that for every $i>ns$ the following holds: $$\frac{b_{i}}{b_{i-1}}\geq q,$$ 
and similarly for every $i<ns$ :
$$\frac{b_{i}}{b_{i-1}}\leq q.$$
From this it follows that $\frac{b_{ns}}{b_{i}}\geq p^{ns-i}$ if $i<ns$ and that $\frac{b_{i}}{b_{ns}}\leq p^{i-ns}$ if $i>ns$. 
From these properties we get an upper bound for $b_i$.
For $i<ns$ we get:
$$b_i=\frac{b_{ns}}{\left( \frac{b_{ns}}{b_{i}}\right)} \leq b_{ns} p^{-(ns-i)}=c p^i (1-p)^{n-i},$$
for $i>ns$ we get:
$$b_i=\frac{b_{ns}}{\left( \frac{b_{i}}{b_{ns}}\right)} \leq b_{ns} p^{i-ns}=c p^i (1-p)^{n-i}.$$
From this we get the desired bound: $a_i \leq \binom{n}{i} c p^i (1-p)^{n-i}$.
\end{proof}

From this we get an upper bound of $P(1)$:
$$P(1)=\sum_{i=0}^n a_i \leq \sum_{i=0}^n c \binom{n}{i}  p^i (1-p)^{n-i} = c,$$
thus $c\geq 1$.

Let $Y$ be a random variable that takes its values from $\{0,1,\dots ,n\}\times \{0,1\}$ with the following probabilities:

$$\P(Y=(k,\epsilon))=\left\{ \begin{array}{lc} \frac{a_k}{c} & \mbox{if}\ \epsilon=1, \\  \binom{n}{k}p^k (1-p)^{n-k}-\frac{a_k}{c} & \mbox{if}\  \epsilon=0. \end{array} \right.$$
Then $Y=(X,Z)$.  Clearly, $Y$ is well-defined since $a_k\leq c\binom{n}{k}p^k (1-p)^{n-k}$. Let $A$ be the event that $Z=1$. Then $A$ has probability $\frac{1}{c}$. Let $X|A$ be the random variable that we condition $X$ on $A$. Clearly,  $X|A$ has distribution $(a_i)$, and $X|\overline{A}$ has distribution $(\binom{n}{i} c/(c-1) p^i (1-p)^{n-i}-a_i/(c-1))$ .
The following holds true:
\begin{itemize}
    \item $X$ has binomial distribution $\mathrm{Bin}(n,p)$,
    \item $a_{ns}= \P[X=ns|A]$,
    \item $\P[X=ns|\overline A]=0$,
    \item $\P[A|X=ns]=1$.
\end{itemize}
All these statements are clear from the definition of $Y$ and $X$, and the fact that $a_{ns}=c\binom{n}{ns}p^{ns}(1-p)^{n-ns}$ by the definition of $p$ and $c$.

Now the heart of our argument is the following general lemma.

\begin{Lemma} \label{main lemma} Let $X$ be a random variable with binomial distribution $\mathrm{Bin}(n,p)$. Suppose that for an event $A$ we have $sn=\E [X|A]$ and $\P[A|X=ns]=1$. Then
$$ \P [X=ns|A]  \geq \binom{n}{ns}(s^s(1-s)^{1-s})^n . $$
\end{Lemma}

\begin{proof}
By the Bayes theorem  we only have to prove the following:
$$ \frac{\P [X=ns]}{\P[A]}  \geq \binom{n}{ns}(s^s(1-s)^{1-s})^n,$$
$$ \frac{\binom{n}{ns}(p^s(1-p)^{1-s})^n }{\P[A]}  \geq \binom{n}{ns}(s^s(1-s)^{1-s})^n  ,$$
$$\frac{(p^s(1-p)^{1-s})^n}{(s^s(1-s)^{1-s})^n} \geq \P[A] .$$
We know that for any $t>0$:
$$\P[A] \leq \frac{\E[e^{tX}]}{\E[e^{tX}|A]}.$$
As $e^{tx}$ is a convex function we know that:
$$\E[e^{tX}|A]\geq e^{t\E[X|A]} =e^{tns},$$
and by that:
$$\P[A] \leq \frac{\E[e^{tX}]}{e^{tns}}.$$
By the linearity of the expected value:
$$\P[A] \leq \left(\frac{\E[e^{tX_0}]}{e^{ts}}\right)^n,$$
where $X_0$ is a Bernoulli random variable with $p$ probability.
If we expand the expected value, we get
$$\left(\frac{\E[e^{tX_0}]}{e^{ts}}\right)^n=\left((1-p)e^{-ts}+pe^{t-ts}\right)^n.$$
We minimize the $n$-th root of this in $t$ by taking its derivative:
$$\frac{\partial}{\partial t}((1-p)e^{-ts}+pe^{t-ts})=p(1-s)e^{t(1-s)}-s(1-p)e^{-ts}=0,$$
$$e^t=\frac{(1-p)s}{(1-s)p}.$$
By this our bound is:
$$\left((1-p)e^{-ts}+pe^te^{-ts}\right)^n=\left((1-p+pe^t)e^{-ts}\right)^n=\left(\left(1-p+\frac{(1-p)s}{(1-s)}\right)\left(\frac{(1-s)p}{(1-p)s}\right)^s\right)^n.$$
Which is exactly the desired bound.

\end{proof}
As $\mathbb P[X=ns|A]=a_{ns}$ this completes our proof.
\end{proof}
\begin{Rem}

In Lemma \ref{main lemma} the expression
$$\frac{(p^s(1-p)^{1-s})^n}{(s^s(1-s)^{1-s})^n} \geq \P[A]$$
can be expressed as

$$e^{-D_1(P||Q)} \geq e^{-D_\infty(Q||P)}$$
where $D_\alpha$ is the Rényi divergence (\cite{van2014renyi}), and where $P$ is the distribution of $X$ and $Q$ is the distribution of $X$ if we restrict it to $A$. 
\smallbreak
This means that Lemma \ref{main lemma} is equivalent to the following statement:
If $P=\mathrm{Bin}(n,p)$ and $\E[Q]=\E[P]$ then 
$$D_1(P||Q) \leq D_\infty(Q||P)$$

This statement can be generalized to handle sums of variables with Bernoulli distributions (which may differ in each summand), not just variables with binomial distributions. 
\end{Rem}
\bigbreak

Now we are ready to prove Theorem~\ref{capacity}.

\begin{proof}[Proof of Theorem~\ref{capacity}]

Let us show that our previous theorem implies this one. First we will show that it implies the following lemma:

\begin{Lemma}
For a Lorentzian polynomial  $p(y,z)=\sum_{i=0}^na_iz^iy^{n-i}$ with non-negative coefficients:
$$a_k\geq \binom{n}{k}\left(\frac{k}{n}\right)^k\left(\frac{n-k}{n}\right)^{n-k}\inf_{t>0}\frac{p(1,t)}{t^k}.$$
\end{Lemma}
\begin{proof}

For a $t>0$ let us consider the probability distribution $q_j=\frac{a_jt^j}{p(t)}$. Then $q(y,z):=\sum_{j=0}^nq_jy^{n-j}z^j$ is still a Lorentzian polynomial. Choose $t_k$ in such a way that $\sum_{j=0}^njq_j=k=ns$, i.e. $s=\frac{k}{n}$.

In this case the coefficients of the polynomial $q(1,z)$ satisfy the requirements of Theorem \ref{thm}, as of Example 2.3 in \cite{branden2020lorentzian} and Theorem 2.25.

Let us apply our theorem. Then
$$\frac{a_kt_k^k}{p(t_k)}=q_k\geq \binom{n}{k}\left(\frac{k}{n}\right)^k\left(\frac{n-k}{n}\right)^{n-k}.$$
In other words,
$$a_k\geq \binom{n}{k}\left(\frac{k}{n}\right)^k\left(\frac{n-k}{n}\right)^{n-k}\frac{p(t_k)}{t_k^k}\geq \binom{n}{k}\left(\frac{k}{n}\right)^k\left(\frac{n-k}{n}\right)^{n-k}\inf_{t>0}\frac{p(t)}{t^k}.$$
\end{proof}

This implies our statement as if we take the polynomial $p(y,z)=P(yx_1,\dots ,z,\dots ,yx_m)$ we get a Lorentzian polynomial for any $\underline{x}^*>0$ vector by Theorem 2.10 in \cite{branden2020lorentzian} where $\underline{x}^*$ is the vector $\underline{x}$ with the $i$-th coordinate ommited.
If we use our Lemma on this polynomial we get
$$a_k \geq  \binom{n}{k}\left(\frac{k}{n}\right)^k\left(\frac{n-k}{n}\right)^{n-k}\inf_{t>0}\frac{p(1,t)}{t^k}.$$
As $a_k=\left.\frac{\partial P}{\partial x_i}\right|_{x_i=0}$ is a polynomial in $\underline{x}^*$, and for all $\underline{x}^*>0$ the inequality holds we get that

$$\frac{1}{k!} \cdot \Cap _{\underline{\alpha}^*}\left(\left.\frac{\partial P}{\partial x_i^k}\right|_{x_i=0}\right)  \geq \Cap _{\underline{\alpha}^* }\bigg(\Cap _k(P) \cdot  \binom{n}{k} \bigg( \frac{k}{n} \bigg)^k \bigg( \frac{n-k}{n} \bigg)^{n-k}\bigg)  .$$

As $k=\alpha_i$ this is exactly the statement of the theorem.
\end{proof}
\begin{Cor}
Let $P(x_1,\dots ,x_n)$ be a Lorentzian polynomial. Suppose that the total degree of $P$ is $d$. Let $\underline{r}=(r_1,\dots ,r_n)$ and $a_{\underline{r}}$ be the coefficient of $\prod_{i=1}^nx_i^{r_i}$ in $P$. Then
$$a_{\underline{r}}\geq \prod_{i=1}^n\binom{d}{r_i}\bigg(\frac{r_i}{d}\bigg)^{r_i}\bigg(\frac{d-r_i}{d}\bigg)^{d-r_i}\mathrm{cap}_{\underline{r}}(P).$$
\end{Cor}
\begin{proof}
To prove this we will use the theorem repeatedly on a polynomial until there are no variables left resulting in a constant. To apply our theorem this way we need to prove that the polynomial $\left.\frac{\partial P}{\partial x_i^k}\right|_{x_i=0}$ still satisfies the requirements of our theorem, or in other words, that it is still Lorentzian.
From Corollary 2.11 in \cite{branden2020lorentzian} it is clear that $\frac{\partial P}{\partial x_i^k}$ is Lorentzian. If we take the matrix $A$ of size $n\times n$ which is identical to the identity matrix save for the row corresponding to coordinate $i$ where it is zero everywhere and take $\frac{\partial P}{\partial x_i^k}(A\underline{x})$ then it will still be Lorentzian by Theorem 2.10 in \cite{branden2020lorentzian} and it will be equal to our polynomial as $A$ take $(x_1\ldots ,x_i,\ldots,x_n)$ to $(x_1\ldots ,0,\ldots,x_n)$.
\end{proof}

\section*{Acknowledgements}
The author would like to thank Péter Csikvári for his help as supervisor. The author would like to thank Jonathan Leake noticing a mistake in a previous version.

\bibliography{biblio}
\bibliographystyle{plain}

\end{document}